\documentclass[
final
]{dmtcs-episciences}

\usepackage[utf8]{inputenc}
\usepackage{subfigure}
\usepackage{amsmath}

\usepackage[round]{natbib}

\newtheorem{theorem}{Theorem}

\author[Bansari. J. Rayjada et. al]{Bansari. J. Rayjada\affiliationmark{1}\thanks{Corresponding author.} 
  , \and Jekil. A. Gadhiya\affiliationmark{1},
  \and Mahadityasinh. A. Sarvaiya\affiliationmark{1}}
\title[Cordial Labeling of Goldberg Snark and its related Graphs]{Cordial Labeling of Goldberg Snark and its related Graphs}
\affiliation{
  Department of Mathematics, Marwadi University, Rajkot-360007, India.}
\begin{document}
\publicationdata
{vol. 25:3 special issue for main purpose}
{2022}
{1}
{10.46298/dmtcs.10472}
{2022-12-3; None}
{2023-1-1}
\maketitle
\begin{abstract}
In graph theory, a Snark is a connected, bridgeless, Cubic graph that cannot be edge-colored with only three colors. Additionally, to avoid some trivial cases, a Snark is typically required to have a girth of minimum five and a cyclic connectivity of minimum four. In this paper, we investigate the Cordial labeling, for one of the modified structures of Snark graph which is known as Goldberg Snark graph. Moreover, a few special forms of Goldberg Snark graph also admit the Cordial labeling.
\end{abstract}

\section{Introduction}
The importance of Snarks is underscored by the relationship of the four colored theorem, one of whose equivalent formulations states that every Snark must be a non-planar graph. The existence of an infinite number of Snarks has been noted, and they remain a subject of recent research due to their complicated and usually elusive structural properties (\cite{4forbes2012snark}).

The study of Snarks traces its origins to the work of Peter G. Tait in 1880, who explored their connection to the four color theorem. However, the word ``Snark" was later taken forward by Martin Gardner in the year 1976 (\cite{11tait1880remarks}). Except their role in coloring problems, Snarks are closely connected to different challenging conjectures in Graph Theory. For example, the cycle double cover conjecture and the 5 -flow conjecture. Beyond extensive research, the nature of Snarks continues to pose important challenges, with their exact properties still not fully understood. The Snark conjecture, proposed by W. T. Tutte, further gives the importance by hypothesizing that every Snark contains a Petersen graph as a minor (\cite{1brinkmann2013generation}).

The Goldberg Snark graphs belong to the family of Snarks first identified by Goldberg in 1981. The Goldberg Snark is made by a series of systematic modifications and followed by the changes brought in the Petersen graph. Originally, the Petersen graph has 10 vertices and 15 edges which forms the base for this construction. The modification process of Petersen graph includes removal of two vertices and their connected edges and reducing the graph to 8 vertices. Even after the reduction and modification, the graph still pertains to its cubic nature and also maintains some of its important features from the original Petersen graph (\cite{7goldberg1981construction}).

\section{Preliminaries}
\subsection{Graph}
A graph $G=(V(G), E(G))$ has two sets, $V(G)=\left\{v_{1}, v_{2}, \ldots, v_{n}\right\}$ known as vertex set of $G$, and $E(G)=\left\{e_{1}, e_{2}, \ldots, e_{n}\right\}$ known as edge set of $G$. Elements of $V(G)$ and $E(G)$ are known as vertices and edges respectively (\cite{3deo2016graph}).

\subsection{Graph Labeling}
The graph labeling is an assignment of labels traditionally shown by integers, to the vertices or to the edges or both, subject to several conditions (\cite{6gallian2022dynamic}).

\subsection{Cordial Labeling}
Consider a graph $G=(V, E)$. Consider a vertex labeling function $f: V(G) \rightarrow\{0,1\}$ which induces an edge labeling function $f^{*}: E(G) \rightarrow\{0,1\}$ defined as $f^{*}(u v)= |f(u)-f(v)|$. Such a labeling function $f$ is known as Cordial labeling of graph G, if it satisfies the conditions $\left|v_{f}(0)-v_{f}(1)\right| \leq 1$ and $\left|e_{f}(0)-e_{f}(1)\right| \leq 1$, where $v_{f}(i)$ represents the number of vertices of G labeled using $f$ and $e_{f}(i)$ represents the number of edges of $G$ labeled using $f^{*} ; i \in\{0,1\}$ (\cite{2cahit1987cordial}).

\subsection{Edge coloring}
Edge coloring of a graph is an assignment of colors to the edges of the graph so that no two adjacent edges have the same color (\cite{3deo2016graph}).

\subsection{Duplication of vertex in a graph}
The duplication of a vertex v in a graph $G$ generates a graph $G^{\prime}$, which is secured by adding a vertex $v^{\prime}$ to the graph $G$, satisfying the property $N(v)=N\left(v^{\prime}\right)$ (\cite{9kaneria2014graceful}).

\subsection{Edge deletion in a graph}
Let $G=(V, E)$ be a graph and $F \subseteq E$. Then the graph obtained by excluding $F$ from $G$ denoted by $G / F$, is the subgraph of $G$ consisting the same vertices as that of $G$, but with all the elements of $F$ excluded. i.e. $G / F=(V, E / F)$ (\cite{3deo2016graph}).

\subsection{Path Union of a Graph}
Consider $k$ copies $G_{1}, G_{2}, G_{3}, \ldots, G_{k}$ of any graph $G$. Then the Path union of $G$ is obtained by adding an edge from $G_{i}$ to $G_{i+1}$, where $i \in\{1,2,3, \ldots, n-1\}$ (\cite{9kaneria2014graceful}).

\subsection{Open Star of a Graph}
A graph obtained by exchanging each vertex of $K_{1, t}$, except the apex vertex by the graphs $G_{1}, G_{2}, \ldots G_{n}$ is called as open star of graph which is shown by $S\left(G_{1}, G_{2}, \ldots, G_{n}\right)$. If we reconsider each vertex of $K_{1, t}$ excluding the apex vertex by graph $G$. i.e. $G_{1}=G_{2}=\cdots G_{n}$. Open star of graph is denoted by $S(t . G)$ (\cite{9kaneria2014graceful}).

\subsection{One point union of $\mathbf{t}$ copies of a path}
A graph $G$, obtained by exchanging each edge of $K_{1, t}$ by a path $P_{n}$ of length $n$ on $n+1$ vertices, is known as a one point union of $t$ copies of path $P_{n}$, which is shown as $P_{n}^{t}$ (\cite{5gadhiya2020some}).

\subsection{One point union of path graphs}
A graph $G$, obtained by replacing every vertex of $P_{n}^{t}$, except the central vertex, by the graphs  $G_{1}, G_{2}, \ldots G_{t n}$, is called one point union for paths of graphs and is shown by $P_{n}^{t}\left(G_{1}, G_{2}, \ldots G_{t n}\right)$. If all these graphs are identical, say $G$, then the notation gets changed to $P_{n}^{t}\left(t_{n}, G\right)$ (\cite{5gadhiya2020some}).

\subsection{Base Graph}
A Base graph is the main former graph from which all other graphs are originated or are formed by operations like edge contraction, deletion, extension, etc (\cite{3deo2016graph}).

\subsection{Cubic Graph}
A Cubic graph is a graph in which all the vertices have degree three. In other words, a Cubic graph is a 3-regular graph (\cite{3deo2016graph}).

\subsection{Petersen Graph}
The Petersen graph is an undirected graph that contains ten vertices and fifteen edges as shown below (\cite{3deo2016graph}).

\begin{figure}[htbp]
\centering
\begin{minipage}{.5\textwidth}
  \centering
  \includegraphics[scale=0.7]{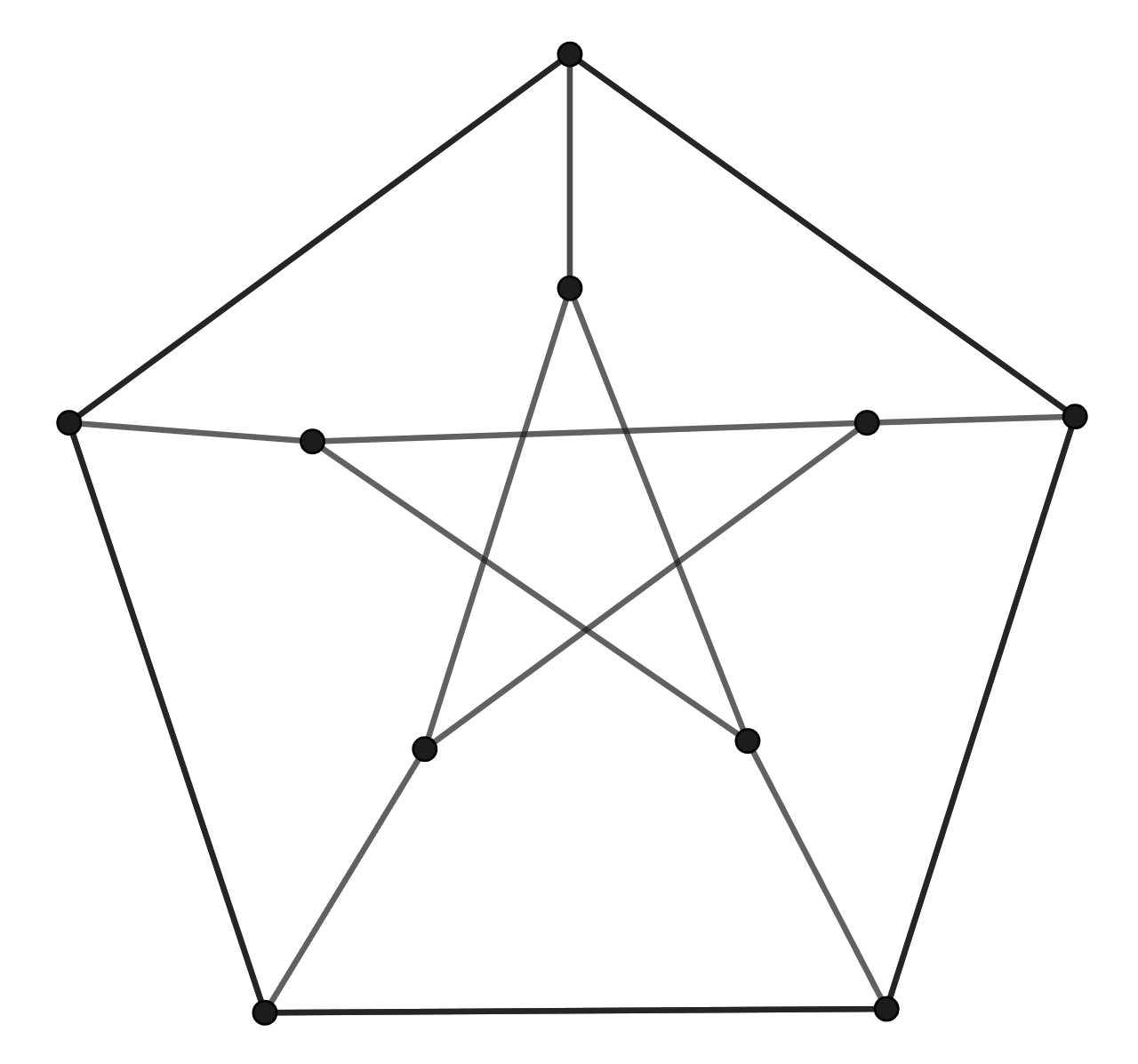}
  \caption{Petersen Graph\label{1}}
\end{minipage}%
\end{figure}

\subsection{Snark Graph}
A Snark is a connected, bridgeless, Cubic graph that cannot be edge-colored with only three colors. Additionally, to avoid some trivial cases, a Snark is typically required to have a girth of minimum five and a cyclic connectivity of minimum four (\cite{4forbes2012snark}).

\subsection{Goldberg Snark Graph}
A Goldberg Snark is a type of nontrivial, bridgeless, Cubic graph that is 4-chromatic and cannot be edge-colored with fewer than four colors, thus serving as a counterexample to the conjecture that every Cubic graph is 3-edge-colorable. To be More specific, the Goldberg Snark is a particular example within the broader class of Snarks, which are graphs that do not admit a 3-edge-coloring (\cite{7goldberg1981construction}).

\begin{figure}[htbp]
\centering
\begin{minipage}{.5\textwidth}
  \centering
  \includegraphics[scale=0.2]{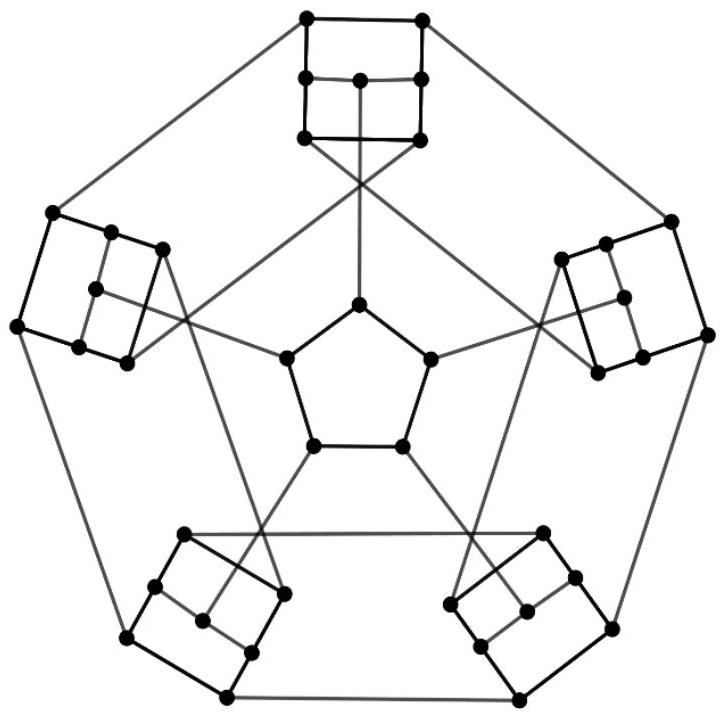}
  \caption{ Goldberg Snark $G_{5}$}\label{2}
\end{minipage}%
\begin{minipage}{.5\textwidth}
  \centering
  \includegraphics[scale=0.2]{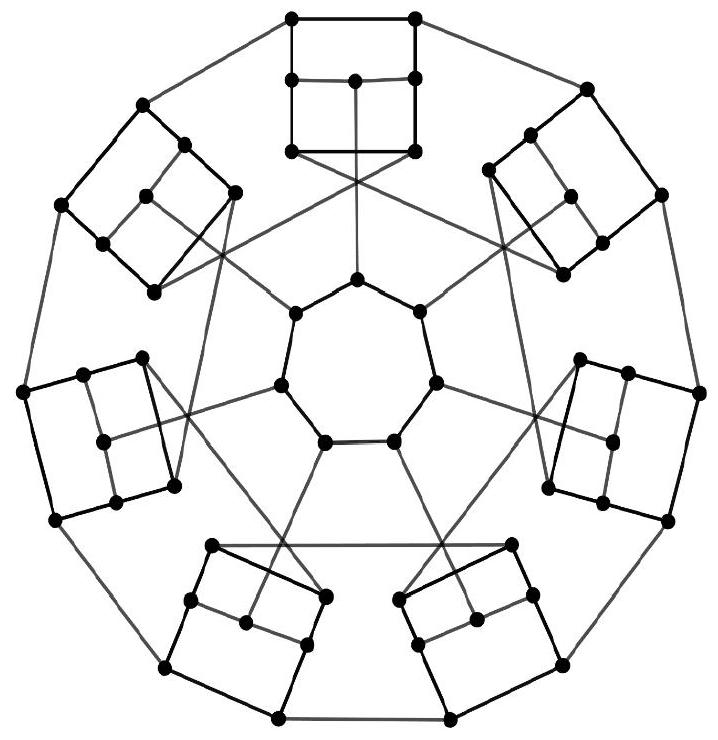}
  \caption{ Goldberg Snark $G_{7}$}\label{3}
\end{minipage}
\end{figure}

\section*{Augmentation of Goldberg Snark Graph}
\subsection*{Isomorphism of subgraphs $\mathbf{H}_\mathbf{i}$ within the Goldberg Snark graph $\mathbf{G_n}$.}

\begin{figure}[htbp]
\centering
\begin{minipage}{.5\textwidth}
  \centering
  \includegraphics[scale=0.25]{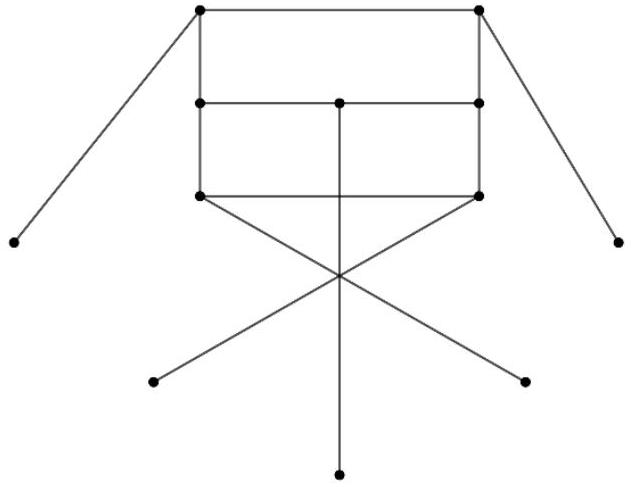}
\caption{ Subgraph $H_i$ of $G_{n}$}\label{4}
\end{minipage}%
\begin{minipage}{.5\textwidth}
  \centering
 \includegraphics[scale=0.25]{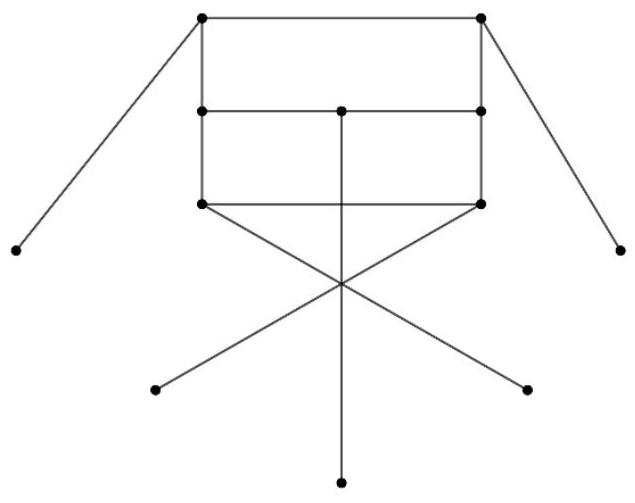}
\caption{ Subgraph $H_j$ of $G_{n}$}\label{5}
\end{minipage}
\end{figure}

\noindent
Observe that the subgraphs $H_{i}$ and $H_{j}$; where $1 \leq i, j \leq n(i \neq j)$ of $G_{n}$ as shown in the above figure, are isomorphic to each other. This property of $G_{n}$ will be put to use, to simplify the labeling pattern of the vertices for a few chosen graphs in some results going forward.

\vspace{6cm}

\section{Goldberg Snark Graph and its Construction}
\subsection{Construction of Snark Graph}
The Snark graph is a specific class of Cubic graphs. This explicit property makes the Snark graph, a structure of specific interest in both theoretical and applied contexts (\cite{4forbes2012snark}).

\subsection*{Characteristics}
\textbf{Non-trivial: }Snarks are non-trivial, meaning they carry more than a trivial number of vertices; specifically, they contain more than four vertices.
\\
\\
\textbf{Cardinality:} The cardinality of the vertex set $V$ and the edge set $E$ for Snark graph can be obtained by its property that they are Cubic graphs. Specific details are as follows.
\\
\\
\textbf{(i) Vertex Set ( $\mathbf{V}$ ):}\\
Even Number of Vertices: As Snarks are Cubic graphs, the total number of vertices ( $|V|$ ) is always even. This is due to a Cubic graph, the sum of the degrees of all the vertices is equal to $3 \times|v|$, which must be an even number (hence $|V|$ must be even).\\
\\
\textbf{(ii) Edge Set (E):}\\
Half the Degree Sum: The number of edges in a Cubic graph which can be calculated from the number of vertices. So, each vertex has degree three, and each edge connects two vertices, the number of edges $(|E|)$ is given by $\frac{3 \times|v|}{2}$ (\cite{1brinkmann2013generation}).

\subsection{Construction of Goldberg Snark Graph}
\subsubsection*{Base Structure}
The Base of the construction is the consideration of a basic Cubic graph with vertices $V=\{v_{1}, v_{2}, v_{3}, \ldots, v_{n}\}$ and edges $E=\{e_{1},$ $e_{2},$ $ e_{3},$ $\ldots,$ $e_{m}\}$ (\cite{10lukotka2024circular}).
\\
\subsubsection*{Vertex Duplication and the corresponding Edge Modification} 
For each vertex $v_{i}$ in the base graph, duplicate it to form two new vertices $v_{i}{ }^{\prime}$ and $v_{i}{ }^{\prime \prime}$. The original edges $e_{i}$ of the base graph are then reconnected such that:\\
Each edge $e_{i}=\left(v_{i}, v_{j}\right)$ in the base graph is replaced by a pair of edges ( $v_{i}{ }^{\prime}, v_{j}{ }^{\prime \prime}$ ) and ( $v_{i}{ }^{\prime \prime}, v_{j}{ }^{\prime}$ ) in the Goldberg Snark graph.\\[0pt]
This modification makes sure that each vertex in the new graph maintains the cubic property, having exactly three edges incident upon it (\cite{7goldberg1981construction}).
\\
\subsubsection*{Edge Subdivisions and Additional Vertices}
Subdivide each edge in the graph made in the above modification by introducing a new vertex. For example, if there is edge ( $v_{i}{ }^{\prime}, v_{j}{ }^{\prime \prime}$ ) in the modified graph, then replace it with the two edges $\left(v_{i}{ }^{\prime}, u_{i j}\right)$ and $\left(u_{i j}, v_{j}{ }^{\prime \prime}\right)$ where $u_{i j}$ the new vertex, added specifically for this edge. This step helps in verifying that every edge in the graph has a connected new vertex, which leads to an increase of graph's girth and supporting of the non-3-edge-colorability property (\cite{8jesintha2023new}).
\\
\subsubsection*{Ensuring Bridgeless Property and Iterative Construction process}
To ensure that the graph remains bridgeless, make sure that no single edge removal disconnects the graph. This can be checked repeatedly by confirming that each edge connects distinct cycles or the removal of any edge does not split the graph into disjoint components. The construction process is repeatedly applied, resulting in a group of subgraphs.
\\
Hence, the Goldberg Snark is made by an organized method which includes Vertex Duplication, Edge Modification, and Subdivision, beginning from a Modified Petersen graph. This process creates a unique class of Snark graphs containing $8 n$ vertices and $12 n$ edges, keeping the cubic structure and the essential non-3-edge-colorability which helps in understanding Snarks and its definitions on the basis of this process. The detailed construction makes sure that the result of the graph meets the strict standards of Snarks, which helps in better understanding of Cubic graphs and its chromatic properties (\cite{10lukotka2024circular}).

\section{Main Results}
\begin{theorem}
    The Goldberg Snark ${G}_{{n}}$ admits a Cordial labeling for ${n} \equiv {1}({\operatorname { m o d } 2})$ and ${n} \neq {3}$.
\end{theorem}
\begin{proof}
    Consider a Goldberg Snark $G_{n}$ with $8 n$ vertices and $12 n$ edges built using $n$-copies of the Goldberg Snark graph on eight vertices, where $n$ remains odd. Therefore, $\left|V\left(G_{n}\right)\right|=8 n$ and $\left|E\left(G_{n}\right)\right|= 12 n$.
\\
Consider the vertex set of $G_{n}$ as,\\
$V=\left\{v_{i, j} ; 1 \leq i \leq n\right.$ and $\left.1 \leq j \leq 8 ; n \in N\right\}$.
\\
To get the Cardinality of vertex in the graph, it can be observed that there is a repetitive structure of the graph shown in the figures \ref{6} and \ref{7}.\\
The number of vertices in a Goldberg Snark is eight and this structure is repeated $n$ times. Therefore, the total number of the vertices of the Graph $G_{n}$ is
\begin{equation}
   |V|=n+n+n+n+n+n+n+n=8 n
\end{equation}
Consider the edge set of $G_{n}$ as,
$$
\begin{aligned}
E=\left\{\left(v_{i, \mathrm{j}} v_{i,(\mathrm{j}+2)}\right.\right. & ; 1 \leq i \leq n, j=1,2,3,5) \cup\left(v_{i, \mathrm{j}} v_{i,(\mathrm{j}+5)} ; 1 \leq i \leq n, j=1,3\right) \\
& \cup\left(v_{i, \mathrm{j}} v_{\mathrm{i}, \mathrm{j}+1} ; 1 \leq i \leq n, j=1,4,6\right) \cup\left(v_{i, \mathrm{j}} v_{(i+1), \mathrm{j}} ; 1 \leq i \leq n-1, j=8\right) \\
& \cup\left(v_{n, j} v_{i, j} ; i=1, j=8\right) \cup\left(v_{i, \mathrm{j}} v_{(i+1),(\mathrm{j}+1)} ; 1 \leq i \leq n-1, \mathrm{j}=6\right) \\
& \cup\left(v_{n, \mathrm{j}} \mathrm{v}_{\mathrm{i}, \mathrm{j}+1)} ; \mathrm{i}=1, \mathrm{j}=6\right) \cup\left(v_{i, j} v_{(i+1),(j-2)} ; 1 \leq i \leq n-1, j=4\right) \\
& \left.\cup\left(v_{n, \mathrm{j}} v_{\mathrm{i}, \mathrm{j}-2)} ; i=1, j=4\right)\right\}
\end{aligned}
$$
\\
$|E|=\{(n+n+n+n+n+n+n+n+n)+(n-1)+(1)+(n-1)+(1)+(n- 1)+ (1)\}$ \\ $=9 n+n-1+1+n-1+1+n-1+1=9 n+n+n+n=12 n$
\\
\\
We have already noted that the subgraphs $H_{i}$ and $H_{j}$ of $G_{n}$ are isomorphic to each other. Due to this isomorphic property of all subgraphs $H_{i}$ of the Goldberg Snark graph $G_{n}$, the arrangement of the vertices remains unchanged for each subgraph $H_{i}, i \in\{1,2,3, \ldots, n\}$.
\\
We have defined two separate labeling functions.
\\
\textbf{Labeling pattern-1:}
The Vertex labeling function is defined as
\\
$$
f\left(v_{i, j}\right)=\left\{\begin{array}{l}
0 ; j \equiv 0(\bmod 2) \\
1 ; j \equiv 1(\bmod 2)
\end{array} ; \text { where } 1 \leq i \leq n, 1 \leq j \leq 8\right.
$$
\\
Now, to check the conditions for Cordial labeling for $f$, consider\\
$\left|v_{f}(0)\right|=\left|v_{f}(1)\right|= n+n+n+n=4 n 
\Rightarrow| | v_{f}(0)\left|-\left|v_{f}(1)\right|\right|=|4 n-4 n|=0 \leq 1$\\
$\left|e_{f}(0)\right|=\left|e_{f}(1)\right|=n+n+n+n-1+1+n-1+1+n-1+1=6 n \Rightarrow| | e_{f}(0)\left|-\left|e_{f}(1)\right|\right|=|6 n-6 n|=0 \leq 1$
\\
\\
\textbf{Labeling pattern-2:}
The Vertex labeling function is defined as\\
$$
f\left(v_{i, j}\right)=\left\{\begin{array}{l}0 ; j=4,6,7,8 \\ 1 ; j=1,2,3,5\end{array} ;\right. \text{ where }1 \leq i \leq n, 1 \leq j \leq 8$$
\\
For both the above definitions, all the subgraphs $H_{i}, 1 \leq i \leq n$ of $G_{n}$ assume the same vertex labeling patterns repeated $n$ number of times.
\\
Now, to check the conditions for Cordial labeling for $f$, consider\\
$\left|v_{f}(0)\right|=\left|v_{f}(1)\right|= n+n+n+n=4 n 
\Rightarrow| | v_{f}(0)\left|-\left|v_{f}(1)\right|\right|=|4 n-4 n|=0 \leq 1$\\
$\left|e_{f}(0)\right|=\left|e_{f}(1)\right|=n+n+n+n-1+1+n-1+1+n-1+1=6 n \Rightarrow| | e_{f}(0)\left|-\left|e_{f}(1)\right|\right|=|6 n-6 n|=0 \leq 1$
\\
Thus, the labeling function satisfies the conditions for Cordial labeling of $G$. Therefore, $G$ admits a Cordial Labeling. Hence, $G$ is a Cordial Graph.
\\
Consider the following two illustrations, one each for the two labeling functions patterns.

\clearpage
\subsubsection*{Illustration 1: A Cordial Labeling of Goldberg Snark $\mathbf{G}_{7}$.}
\begin{figure}[htbp]
\begin{center}
  \includegraphics[scale=0.3]{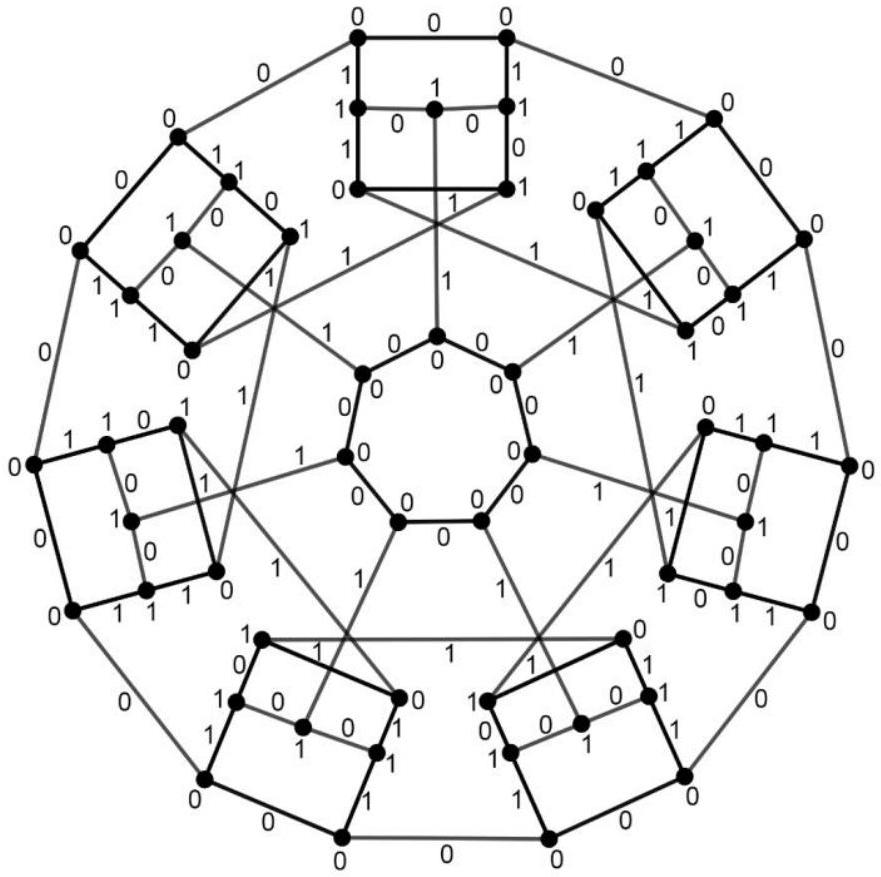}
\caption{ A Cordial Labeling of $G_7$}\label{6}
\end{center}
\end{figure}
\clearpage
\subsubsection*{Illustration 2: A Cordial Labeling of Goldberg Snark $\mathbf{G_5}$.}
\begin{figure}[htbp]
\begin{center}
  \includegraphics[scale=0.3]{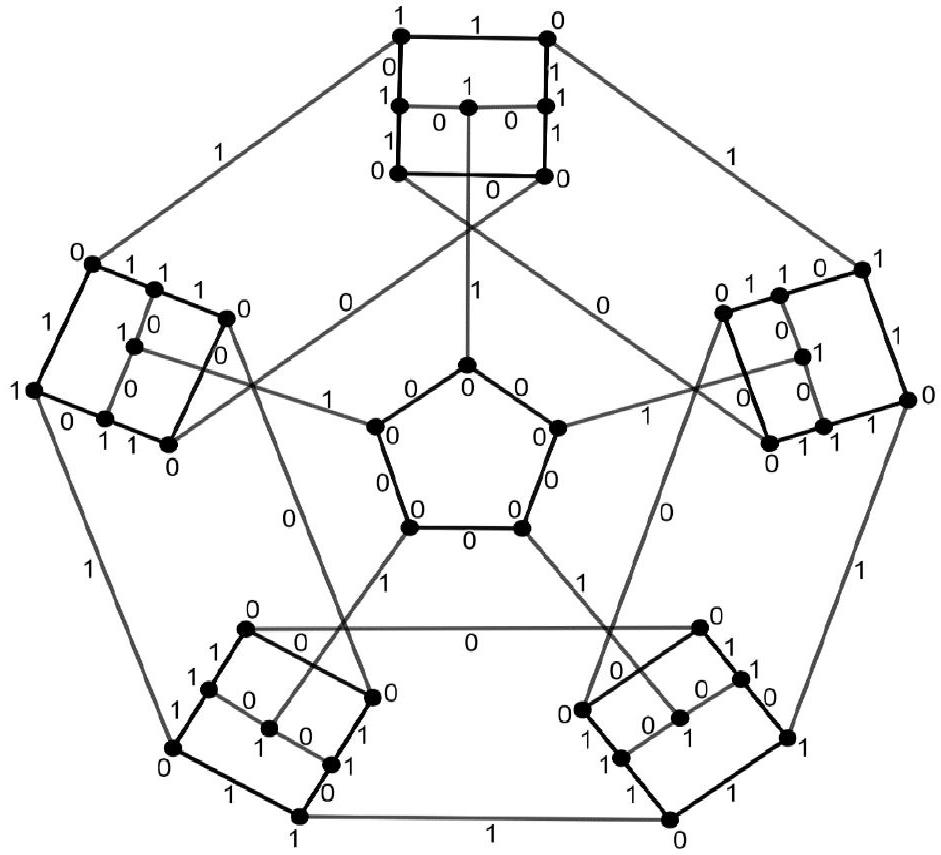}
\caption{ A Cordial Labeling of $G_5$}\label{7}
\end{center}
\end{figure}
\end{proof}

\begin{theorem}
   Path Union of ${m}$-copies of Goldberg Snark graph ${G}_{n}$ is a Cordial graph.   
\end{theorem}
\begin{proof}
Let $G$ be a graph obtained by joining $m$-copies of the Goldberg Snark $G_{n}$.\\
Consider the vertex set of G as,
\\
$V=\left\{v_{i, j, k} ; 1 \leq i \leq n, 1 \leq j \leq 8,1 \leq k \leq m\right\}$\\
where $k$ represents the number of paths of graph $G$ and $m$ represents the number of $m$-copies of the Goldberg Snark $G_{n}$ in $G$. Thus, \\ 
$|V(G)|=m\left|V\left(G_{n}\right)\right|$ and \\
$|E(G)|=m\left|E\left(G_{n}\right)\right|+m-1$
\\
To obtain Cordial labeling of Path Union of $m$-copies of the Goldberg Snark $G_{n}$ defined a vertex labeling function $f: V(G) \rightarrow\{0,1\}$ as below.
\\
\\$f\left[v_{(i, j, k)}\right]=\left\{\begin{array}{c}0 ; j \equiv 0(\bmod 2), m \equiv 1,2(\bmod 4) \\ 1 ; j \equiv 1(\bmod 2), m \equiv 1,2(\bmod 4)\end{array} ;\right.$ where $1 \leq i \leq n, 1 \leq j \leq 8$\\
$f\left[v_{(i, j, k)}\right]=\left\{\begin{array}{l}0 ; j=4,6,7,8, m \equiv 3,4(\bmod 4) \\ 1 ; j=1,2,3,5, m \equiv 3,4(\bmod 4)\end{array} ;\right.$ where $1 \leq i \leq n, 1 \leq j \leq 8$\\
By considering the above labeling pattern, we will now check for the conditions of Cordial labeling for $f$.\\
$\left|v_{f}(0)\right|=\left|v_{f}(1)\right|=\left|\frac{m\left|V\left(G_{n}\right)\right|}{2}\right|$\\
$\left|v_{f}(0)-v_{f}(1)\right| \Rightarrow\left|\left|\frac{m\left|V\left(G_{n}\right)\right|}{2}\right|-\left|\frac{m\left|V\left(G_{n}\right)\right|}{2}\right|\right|=0 \leq 1$
\\
\\In order to check the edge labeling condition we consider the following two cases.\\
\\
\textbf{Case-I: $\boldsymbol{m \equiv 0}(\boldsymbol{\operatorname { m o d }}\, 2)$.}\\
$\left|e_{f}(0)\right|=\left|\frac{m\left|E\left(G_{n}\right)\right|}{2}+\left\lceil\frac{k}{2}\right\rceil\right|$ or $\left|e_{f}(0)\right|= \left| \left|e_{f}(1)\right| + 1\right|$\\
$\left|e_{f}(1)\right|=\left|\frac{m\left|E\left(G_{n}\right)\right|}{2}+\left\lfloor\frac{k}{2}\right\rfloor\right|$\\
$\left|e_{f}(0)-e_{f}(1)\right| \Rightarrow\left|\left|\frac{m\left|E\left(G_{n}\right)\right|}{2}+\left\lceil\frac{k}{2}\right\rceil\right|-\left|\frac{m\left|E\left(G_{n}\right)\right|}{2}+\left\lfloor\frac{k}{2}\right\rfloor\right|\right|=1 \leq 1$
\\
\\
In the view of the above labeling pattern, $f$ satisfies the conditions of Cordial labeling for $G$. 
\\
Thus, the graph $G$ is a Cordial graph for even $m$.
\\
\\
\textbf{Case-II: $\boldsymbol{m  \equiv 1}(\boldsymbol{\operatorname { m o d }\, 2)}$.}\\
By considering the above labeling pattern, we will now check for the condition of Cordial labeling for $f$.
\\
$
\left|e_{f}(0)\right|=\left|e_{f}(1)\right|=\left|\frac{m\left|E\left(G_{n}\right)\right|}{2}+\frac{k}{2}\right|
$
\\
$\left|e_{f}(0)-e_{f}(1)\right| \Rightarrow\left|\left|\frac{m\left|E\left(G_{n}\right)\right|}{2}+\frac{k}{2}\right|-\left|\frac{m\left|E\left(G_{n}\right)\right|}{2}+\frac{k}{2}\right|\right|=0 \leq 1$\\
In the view of above labeling pattern, $f$ satisfies the conditions of Cordial labeling for $G$. Therefore, the graph $G$ is a Cordial graph for odd $m$.\\
Hence, the Path Union of $m$-copies of Goldberg Snark graph $G_{n}$ is Cordial graph.

\subsubsection*{Illustration 3: A Cordial Labeling of Path union of four copies of Goldberg Snark $\mathbf{G}_{5}$.}
\begin{figure}[htbp]
\begin{center}
  \includegraphics[width=\textwidth]{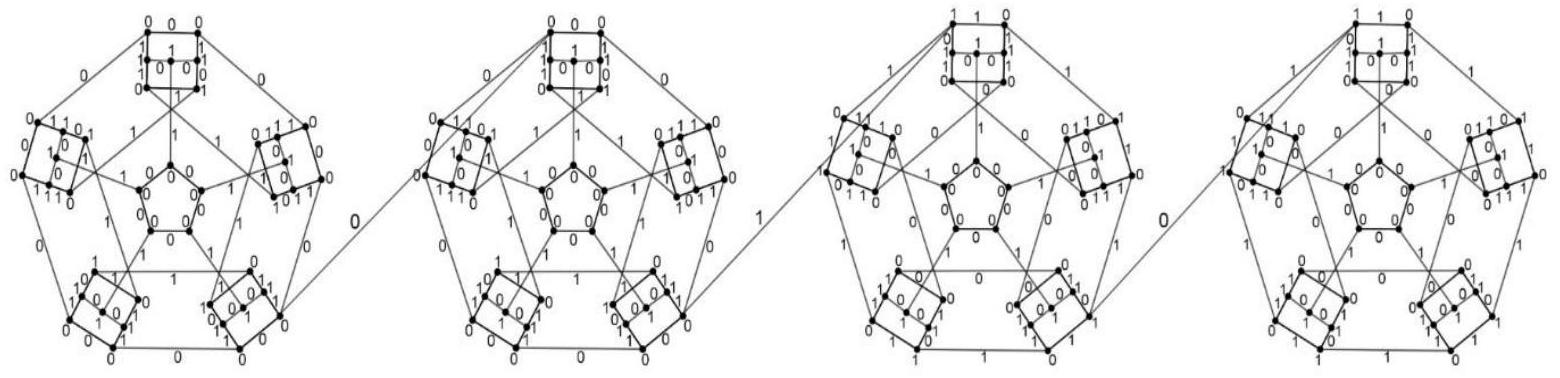}
\caption{ A Cordial Labeling of Path Union of four copies of $G_5$}\label{8}
\end{center}
\end{figure}
\noindent

\clearpage
\subsubsection*{Illustration 4: A Cordial Labeling of Path union of five copies of Goldberg Snark $\boldsymbol{G}_{\mathbf{7}}$.}
\begin{figure}[htbp]
\begin{center}
  \includegraphics[width=\textwidth]{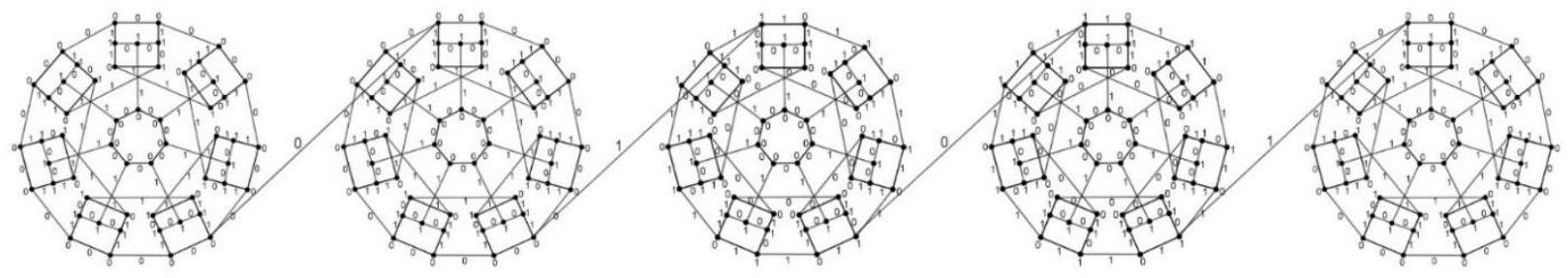}
\caption{ A Cordial Labeling of Path Union of five copies of $G_7$}\label{9}
\end{center}
\end{figure}
\end{proof}
\begin{theorem}
    An Open Star of Graphs ${S}\left({t}, {G}_{{n}}\right)$ is a Cordial Graph.
\end{theorem}
\begin{proof}
  Let $G$ be a graph obtained by replacing each vertex of $k_{1, t}$ except the apex vertex by Goldberg Snark graph $G_{n}$. Let $v_{0}$ be the apex vertex of $k_{1, t}$, i.e. $v_{0}$ is the central vertex of Graph $G$.
\\
Let $V=\left\{v_{i, j, k} ; 1 \leq i \leq n, 1 \leq j \leq 8,1 \leq k \leq t\right\}$ be the vertex set of the Goldberg Snark $G_{n}$ in $G$; where $k$ represents the number of branches of $k_{1, t}$ in $G$.
\\
\\
Assign the label 0 to the central vertex of Graph $G$. Define a vertex labeling function $f: V(G) \rightarrow \{0,1\}$ as below.\\
Consider the following cases:
\\
\\
\textbf{Case-I: $\boldsymbol{t \equiv 0}(\boldsymbol{\operatorname { m o d }} \, 2)$.}\\
$f\left[v_{(i, j, k)}\right]=\left\{\begin{array}{l}0 ; j \equiv 0(\bmod 2), 1 \leq k \leq\left[\frac{t}{2}\right] \\ 1 ; j \equiv 1(\bmod 2), 1 \leq k \leq\left[\frac{t}{2}\right]\end{array} ;\right.$ where $1 \leq i \leq n, 1 \leq j \leq 8$\\
\\
$f\left[v_{(i, j, k)}\right]=\left\{\begin{array}{l}0 ; j=4,6,7,8,\left[\frac{t}{2}\right]<k \leq t \\ 1 ; j=1,2,3,5,\left[\frac{t}{2}\right]<k \leq t\end{array} \quad ;\right.$ where $1 \leq i \leq n, 1 \leq j \leq 8$
\\
\\In the view of above labeling pattern, we now check for the conditions of Cordial Labeling for $f$.\\
$\left|v_{f}(0)\right|=\left|\frac{t\left|V\left(G_{n}\right)\right|}{2}+1\right|$ or $\left|v_{f}(0)\right|=\left|\left|v_{f}(1)\right|+1\right|$\\
$\left|v_{f}(1)\right|=\left|\frac{t\left|V\left(G_{n}\right)\right|}{2}\right|$\\
$\left|v_{f}(0)-v_{f}(1)\right| \Rightarrow\left|\left|\frac{t\left|V\left(G_{n}\right)\right|}{2}+1\right|-\left|\frac{t\left|V\left(G_{n}\right)\right|}{2}\right|\right|=1 \leq 1$\\
$\left|e_{f}(0)\right|=\left|e_{f}(1)\right|=\left|\frac{\left|t E\left(G_{n}\right)\right|}{2}+\frac{k}{2}\right|$\\
$\left|e_{f}(0)-e_{f}(1)\right| \Rightarrow\left|\left|\frac{\left|t E\left(G_{n}\right)\right|}{2}+\frac{k}{2}\right|-\left|\frac{t\left|E\left(G_{n}\right)\right|}{2}+\frac{k}{2}\right|\right|=0 \leq 1$
\\
\\In view of the above labeling pattern, $f$ satisfies the conditions of Cordial labeling for $G$. Thus, the graph $G$ is a Cordial graph for even $t$.  
\\
\\
\textbf{Case-II: $\boldsymbol{t \equiv} \boldsymbol{1}(\boldsymbol{\operatorname { m o d }} 2)$.}\\
$f\left[v_{(i, j, k)}\right]=\left\{\begin{array}{l}0 ; j \equiv 0(\bmod 2), 1 \leq k \leq\left\lceil\frac{t}{2}\right\rceil \\ 1 ; j \equiv 1(\bmod 2), 1 \leq k \leq\left\lceil\frac{t}{2}\right\rceil\end{array} \quad ;\right.$ where $1 \leq i \leq n, 1 \leq j \leq 8$\\
$f\left[v_{(i, j, k)}\right]=\left\{\begin{array}{l}0 ; j=4,6,7,8,\left\lceil\frac{t}{2}\right\rceil<k \leq t \\ 1 ; j=1,2,3,5,\left\lceil\frac{t}{2}\right\rceil<k \leq t\end{array} \quad ;\right.$ where $1 \leq i \leq n, 1 \leq j \leq 8$\\
\\
\\In the view of above labeling pattern, we now check for the conditions of Cordial labeling for $f$.\\
$\left|v_{f}(0)\right|=\left|\frac{t\left|V\left(G_{n}\right)\right|}{2}+1\right|$ or $\left|v_{f}(0)\right|=\left|\left|v_{f}(1)\right|+1\right|$\\
$\left|v_{f}(1)\right|=\left|\frac{t\left|V\left(G_{n}\right)\right|}{2}\right|$\\
$\left|v_{f}(0)-v_{f}(1)\right| \Rightarrow\left|\left|\frac{t\left|V\left(G_{n}\right)\right|}{2}+1\right|-\left|\frac{t\left|V\left(G_{n}\right)\right|}{2}\right|\right|=1 \leq 1$\\
$\left|e_{f}(0)\right|=\left|\frac{\left|t\left(G_{n}\right)\right|}{2}+\left\lceil \frac{k}{2} \right\rceil \right|$ or $\left|e_{f}(0)\right|=\left|\left|e_{f}(1)\right|+1\right|$\\
$\left|e_{f}(1)\right|=\left|\frac{t\left|E\left(G_{n}\right)\right|}{2}+\left\lfloor\frac{k}{2}\right\rfloor\right|$\\
$\left|e_{f}(0)-e_{f}(1)\right| \Rightarrow \left| \left| \frac{t\left|E\left(G_{n}\right)\right|}{2} + \left\lceil \frac{k}{2} \right\rceil \right| - \left| \frac{t\left|E\left(G_{n}\right)\right|}{2} + \left\lfloor \frac{k}{2} \right\rfloor \right| \right| = 1 \leq 1$
\\
\\In the view of above labeling pattern, $f$ satisfies the conditions of Cordial labeling for $G$. Thus, the graph $G$ is a Cordial graph for odd $t$. Therefore, the Open Star of graph $S\left(t, G_{n}\right)$ is a Cordial graph.

\subsubsection*{Illustration 5: A Cordial Labeling of $\mathbf{S(4,G_7)}$.}
\begin{figure}[htbp]
\begin{center}
  \includegraphics[width=\textwidth]{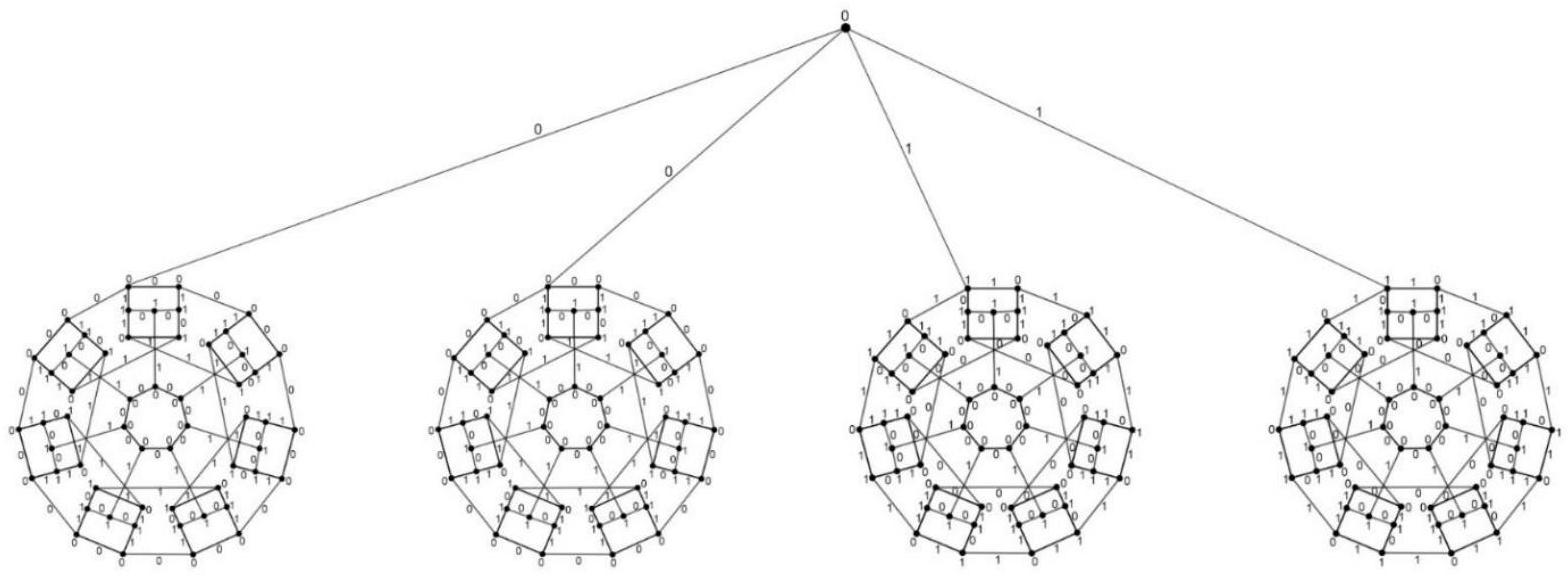}
\caption{ A Cordial Labeling of $\mathbf{S(4,G_7)}$} \label{10}
\end{center}
\end{figure}
\noindent
\clearpage
\subsubsection*{Illustration 6: A Cordial Labeling of $S(3,G_5)$.}
\begin{figure}[htbp]
\begin{center}
  \includegraphics[width=\textwidth]{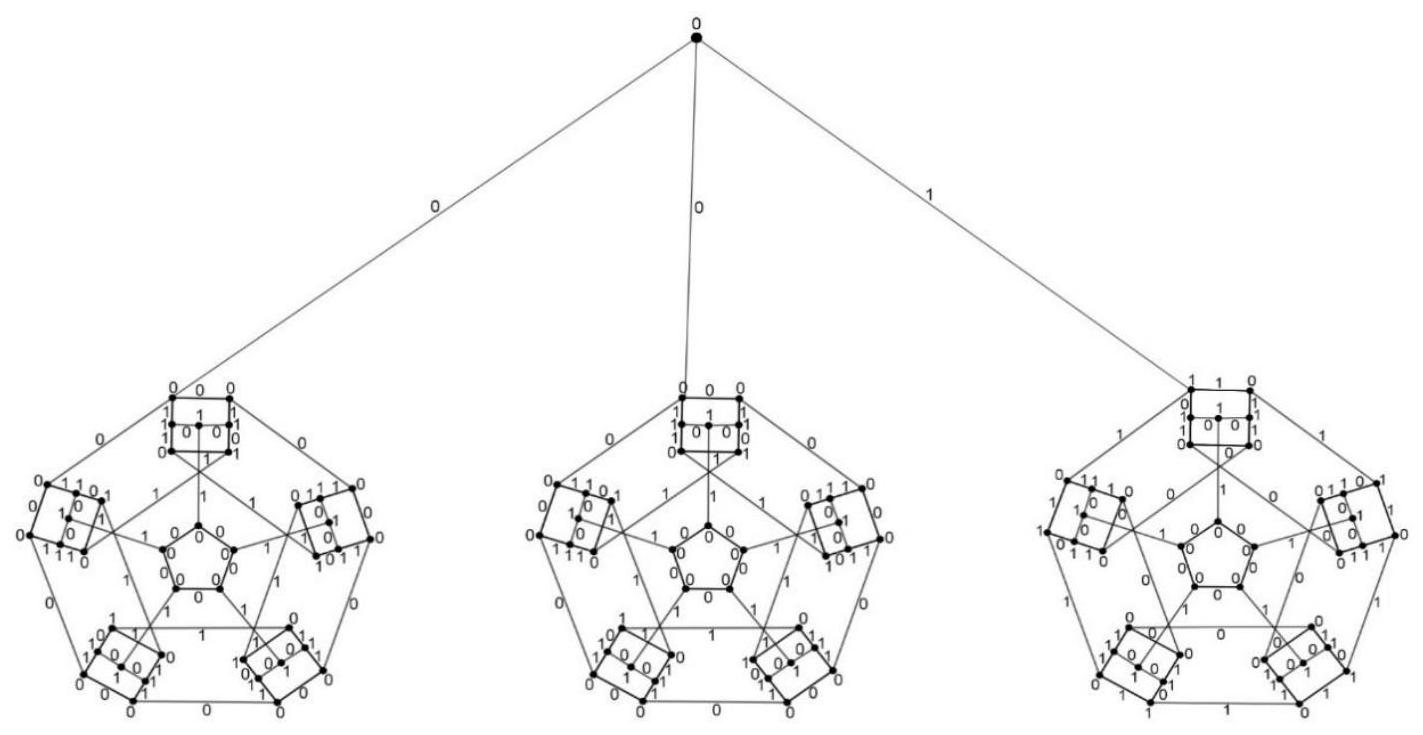}
\caption{ A Cordial Labeling of $S(3,G_5)$} \label{11}
\end{center}
\end{figure}
\end{proof}
\begin{theorem}
    A one point union for a path of a Goldberg Snark $P_{n}^{t}\left(t_{n}, G_{n}\right)$ is a Cordial graph.
\end{theorem}
\begin{proof}
    Consider a graph obtained by replacing each vertex of one point union for a path $P_{n}^{t}$ by the Goldberg Snark graph $\operatorname G_{n}$.
\\
Note that the Goldberg Snark $G_{n}$ is a graph obtained in four layers. Let the vertex sets of the graph $G_{n}$ be $V= \left\{v_{0}, v_{i, j, k, l, m}\right\}$.
\\
\\where $v_{0}$ is the apex vertex; where $l$ represents edge connected to an Open Star graph via apex vertex, $k$ represents the number of copies of Snark $G_{n}$ in Path Union which replace the vertex on path which are connected to apex vertex through edge and $m$ represents the number of copies of path in the Path Union connected to $k$-copies of the Snark graph, which implies $k= m+1$.
\\
where $1 \leq i \leq n, 1 \leq j \leq 8$\\
$k=1,2, \ldots, q ; q \in N $\\
$l=1,2, \ldots, t ; t \in N$\\
$m=1,2, \ldots, p ; p \in N$

\noindent
The inner most layer can be obtained by the Petersen graph. Again, due to the isomorphic property of all subgraphs $H_{r}$ of the Goldberg Snark graph $G_{n}$, the arrangement of the vertices remains unchanged for each subgraph $H_{r}, r \in\{1,2,3, \ldots, n\}$. Now assign a label 0 to the central vertex $v_{0}$. Because of the isomorphic nature of the subgraphs $H_{r}$ of $G_{n}$, and since $l$ is even in this case, these are $l$ subgraphs of $H_{r}$ for $1 \leq r \leq l$. And so we can assign a label 0 to any of the $\frac{l}{2}$ branches of $G_{n}$ and a label 1 to the remaining $\frac{l}{2}$ branches of $G_{n}$.
\\
Define a vertex labeling function $f: V(G) \rightarrow\{0,1\}$ as below. To make Path Union of $G_{n}$, connect the vertices $v_{0}$ and $v_{i, j, k, l, m}$ through an edge. The labeling of the vertices has been given in the following cases:
\\
\\
\textbf{Case I : $\boldsymbol{l \equiv 0}(\boldsymbol{mod 2}), \boldsymbol{k \equiv 0},\, \boldsymbol{1}(\boldsymbol{\operatorname { m o d } 2}),\, \boldsymbol{m \equiv 0},\, \boldsymbol{1}(\boldsymbol{\operatorname { m o d } 2})$}\\
Here, we consider, $1 \leq i \leq n, 1 \leq j \leq 8$\\
$f\left(v_{i, j, k, l, m}\right)=\left\{\begin{array}{l}0 ; j \equiv 0(\bmod 2), 1 \leq l \leq \frac{t}{2}, 1 \leq k \leq q, 1 \leq m \leq p \\ 1 ; j \equiv 1(\bmod 2), 1 \leq l \leq \frac{t}{2}, 1 \leq k \leq q, 1 \leq m \leq p\end{array}\right.$
\\
\\
$f\left(v_{i, j, k, l, m}\right)=\left\{\begin{array}{l}
0 ; j \equiv 0(\bmod 2), \frac{t}{2}<l \leq t, k=\left\lfloor\frac{4 q+1}{2}\right\rfloor, m=\left\lceil\frac{4 p-3}{2}\right\rceil, p ;q=1,2,3, \ldots \\
1 ; j \equiv 1(\bmod 2), \frac{t}{2}<l \leq t, k=\left\lfloor\frac{4 q+1}{2}\right\rfloor, m=\left\lceil\frac{4 p-3}{2}\right\rceil, p ;q=1,2,3, \ldots
\end{array}\right.$
\\
$f\left(v_{i, j, k, l, m}\right)=\left\{\begin{array}{l}
0 ; j=4,6,7,8, \frac{t}{2}<l \leq t, k=\left\lceil\frac{4 q-3}{2}\right\rceil, m=\left\lfloor\frac{4 p+1}{2}\right\rfloor, p ;q=1,2,3, \ldots \\
1 ; j=1,2,3,5, \frac{t}{2}<l \leq t, k=\left\lceil\frac{4 q-3}{2}\right\rceil, m=\left\lfloor\frac{4 p+1}{2}\right\rfloor, p ;q=1,2,3, \ldots
\end{array}\right.$
\\
\\
By considering the above labeling pattern, we will now check for the conditions of Cordial labeling for $f$.
\\
$$
\begin{aligned}
& \left|v_{f}(0)\right|=\left|\frac{n t\left|V_{G_{n}}\right|}{2}+1\right| \text { or }\left|v_{f}(0)\right|=\left|\left|v_{f}(1)\right|+1\right| \\
& \left|v_{f}(1)\right|=\left|\frac{n t\left|V_{G_{n}}\right|}{2}\right| \\
& \left|\left|v_{f}(0)\right|-\left|v_{f}(1)\right|\right| \Rightarrow\left|\left|\frac{n t\left|V_{G_{n}}\right|}{2}+1\right|-\left|\frac{n t\left|V_{G_{n}}\right|}{2}\right|\right\rvert\,=1 \leq 1 \\
& \left|e_{f}(0)\right|=\left|e_{f}(1)\right|=\left|\frac{n t\left|E_{G_{n}}\right|}{2}+\frac{n t}{2}\right| \\
& \left|\left|e_{f}(0)\right|-\left|e_{f}(1)\right|\right| \Rightarrow\left|\left|\frac{n t\left|E_{G_{n}}\right|}{2}+\frac{n t}{2}\right|-\left|\frac{n t\left|E_{G_{n}}\right|}{2}+\frac{n t}{2}\right|\right\rvert\,=0 \leq 1
\end{aligned}
$$
\\
In the view of above labeling pattern, $f$ satisfies the conditions of Cordial labeling for $G$. Thus, the graph $G$ is Cordial for even $l, k$ and odd $m$. Hence, the graph $G$ is Cordial for even $l, m$ and odd $k$.
\\
\\
\textbf{Case II: $\boldsymbol{l} \boldsymbol{\equiv} \boldsymbol{1}\boldsymbol{(\bmod 2),} \, \boldsymbol{k \equiv }\boldsymbol{0}, \boldsymbol{1}\boldsymbol{(\bmod 2),} \boldsymbol{m \equiv }\boldsymbol{1}, \boldsymbol{0(\bmod 2)}$}\\
Here, we consider, $1 \leq i \leq n, 1 \leq j \leq 8$\\
$f\left(v_{i, j, k, l, m}\right)=\left\{\begin{array}{l}0 ; j \equiv 0(\bmod 2), l=\left\lfloor\frac{4 t+1}{2}\right\rfloor, 1 \leq k \leq q, 1 \leq m \leq p, t=1,2,3, \ldots \\ 1 ; j \equiv 1(\bmod 2), l=\left\lfloor\frac{4 t+1}{2}\right\rfloor, 1 \leq k \leq q, 1 \leq m \leq p, t=1,2,3, \ldots\end{array}\right.$\\
$f\left(v_{i, j, k, l, m}\right)=\left\{\begin{array}{l}0 ; j \equiv 0(\bmod 2), l=\left\lceil\frac{4 t-3}{2}\right\rceil, k=\left\lceil\frac{4 q-3}{2}\right\rceil, m=\left\lfloor\frac{4 p+1}{2}\right\rfloor, t; p; q=1,2,3, \ldots \\ 1 ; j \equiv 1(\bmod 2), l=\left\lceil\frac{4 t-3}{2}\right\rceil, k=\left\lceil\frac{4 q-3}{2}\right\rceil, m=\left\lfloor\frac{4 p+1}{2}\right\rfloor, t; p; q=1,2,3, \ldots\end{array}\right.$\\
$f\left(v_{i, j, k, l, m}\right)=\left\{\begin{array}{l}0 ; j=4,6,7,8, l=\left\lceil\frac{4 t-3}{2}\right\rceil, k=\left\lfloor\frac{4 q+1}{2}\right\rfloor, m=\left\lceil\frac{4 p-3}{2}\right\rceil, t; p; q=1,2,3, \ldots \\ 1 ; j=1,2,3,5, l=\left\lceil\frac{4 t-3}{2}\right\rceil, k=\left\lfloor\frac{4 q+1}{2}\right\rfloor, m=\left\lceil\frac{4 p-3}{2}\right\rceil, t; p; q=1,2,3, \ldots\end{array}\right.$\\
\\
\\
In the view of the above labeling pattern, we now check for the condition of vertex labeling for the cordial labeling for $f$.\\
$\left|v_{f}(0)\right|=\left|\frac{n t\left|V_{G_{n}}\right|}{2}+1\right|$ or $\left|v_{f}(0)\right|=\left|\left|v_{f}(1)\right|+1\right|$\\
$\left|v_{f}(1)\right|=\left|\frac{n t\left|V_{G_{n}}\right|}{2}\right|$\\
$\left|\left|v_{f}(0)\right|-\left|v_{f}(1)\right|\right| \Rightarrow\left|\left|\frac{n t\left|V_{G_{n}}\right|}{2}+1\right|-\left|\frac{n t\left|V_{G_{n}}\right|}{2}\right|\right|=1 \leq 1$
\\
\\In view of the above labeling pattern, we now check for the condition of edge labeling for the cordial labeling for $f$.
\\
\\
\textbf{Sub-Case II (a) : $\boldsymbol{k \equiv} \mathbf{0}(\boldsymbol{\operatorname { m o d } 2}),\, \boldsymbol{m \equiv }\mathbf{1}(\boldsymbol{\operatorname { m o d } 2})$}\\
$\left|e_{f}(0)\right|=\left|\frac{n t\left|E_{G_{n}}\right|}{2}+\frac{n t}{2}\right|$\\
$\left|e_{f}(1)\right|=\left|\frac{n t\left|E_{G_{n}}\right|}{2}+\frac{n t}{2}\right|$\\
$\left|\left|e_{f}(0)\right|-\left|e_{f}(1)\right|\right| \Rightarrow\left|\left|\frac{n t\left|E_{G_{n}}\right|}{2}+\frac{n t}{2}\right|-\left|\frac{n t\left|E_{G_{n}}\right|}{2}+\frac{n t}{2}\right|\right|=0 \leq 1$
\\
\\In view of the above labeling pattern, $f$ satisfies the conditions of the Cordial labeling for $G$. Therefore, the Graph $G$ is Cordial for even $k$ and odd $l$ and $m$.
\\
\\
\textbf{Sub-Case II} $(\boldsymbol{b}): \boldsymbol{k \equiv 1}(\boldsymbol{\operatorname { m o d } 2}),\, \boldsymbol{m \equiv} \mathbf{0}(\boldsymbol{\operatorname { m o d } 2})$
\\
In the view of above labeling pattern, we now check for the condition of Cordial labeling for $f$.\\
$\left|e_{f}(0)\right|=\left|\frac{n t\left|E_{G_{n}}\right|}{2}+\left\lceil\frac{n t}{2}\right\rceil\right|$ \ or $\left|e_{f}(0)\right|=\left|\left|e_{f}(1)\right|+1\right|$\\
$\left|e_{f}(1)\right|=\left|\frac{n t\left|E_{G_{n}}\right|}{2}+\left\lfloor\frac{n t}{2}\right\rfloor\right|$\\
$\left|\left|e_{f}(0)\right|-\left|e_{f}(1)\right|\right| \Rightarrow\left|\left|\frac{n t\left|E_{G_{n}}\right|}{2}+\left\lceil\frac{n t}{2}\right\rceil\right|-\left|\frac{n t\left|E_{G_{n}}\right|}{2}+\left\lfloor\frac{n t}{2}\right\rfloor\right|\right|=1 \leq 1$
\\
\\In the view of above labeling pattern, $f$ satisfies the conditions of the Cordial labeling for $G$. Hence, the Graph $G$ is Cordial for even $m$ and odd $l$ and $k$.
\\
\\Therefore, one point union for a path of a Goldberg Snark $P_{n}^{t}\left(t_{n}, G_{n}\right)$ is a Cordial graph.
\clearpage

\subsubsection*{Illustration 7: A Cordial Labeling of $P_{4}^{4}(4_4,G_7)$}
\vspace{-0.01cm}
\begin{figure}[htbp!]
\begin{center}
  \includegraphics[scale=0.2]{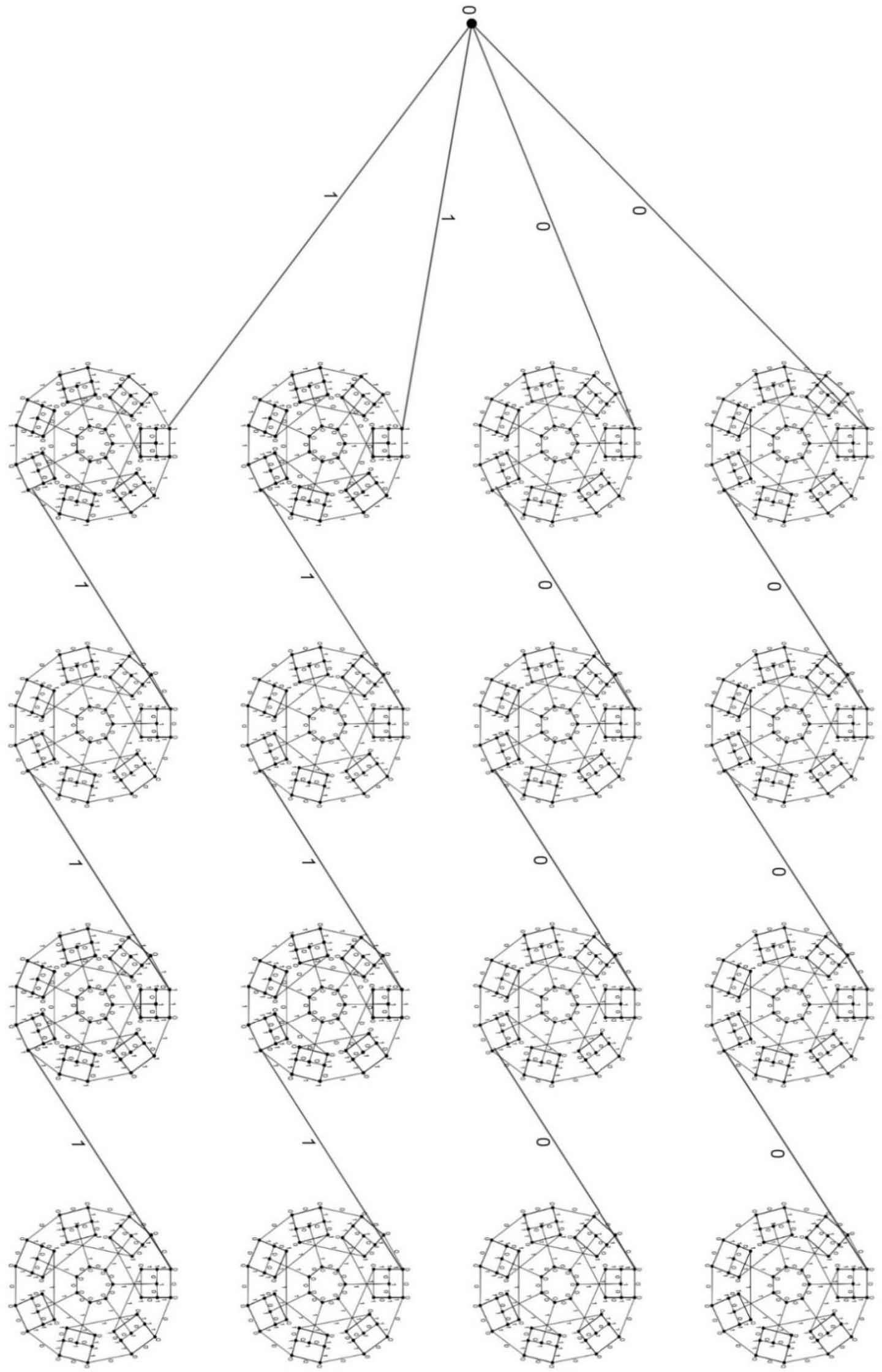}
\caption{A Cordial Labeling of $P_{4}^{4}(4_4,G_7)$}\label{12}
\end{center}
\end{figure}
\vspace{1cm}
 \noindent
\clearpage
\subsubsection*{Illustration 8: A Cordial Labeling of $P_{3}^{4}\left(\mathbf{4}_{3}, G_{5}\right)$.}
\begin{figure}[htbp!]
\begin{center}
  \includegraphics[scale=0.19]{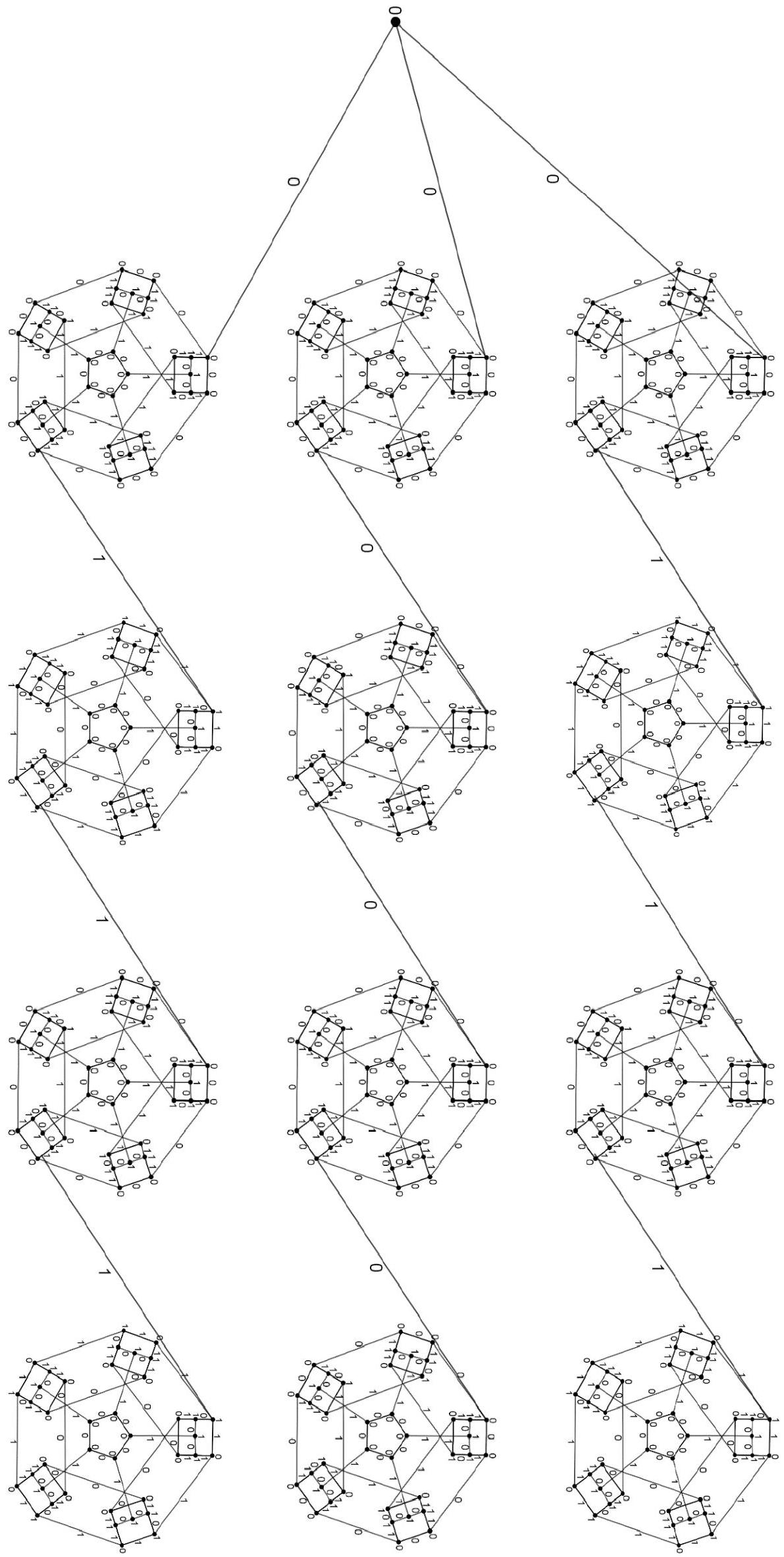}
\caption{A Cordial Labeling of $P_{3}^{4}\left(\mathbf{4}_{3}, G_{5}\right)$}\label{13}
\end{center}
\end{figure}
\vspace{1cm}
\noindent

\clearpage
\subsubsection*{Illustration 9:  A Cordial Labeling of $P_{3}^{3}\left(3_{3}, G_{5}\right)$.}
\begin{figure}[htbp!]
\begin{center}
  \includegraphics[scale=0.22]{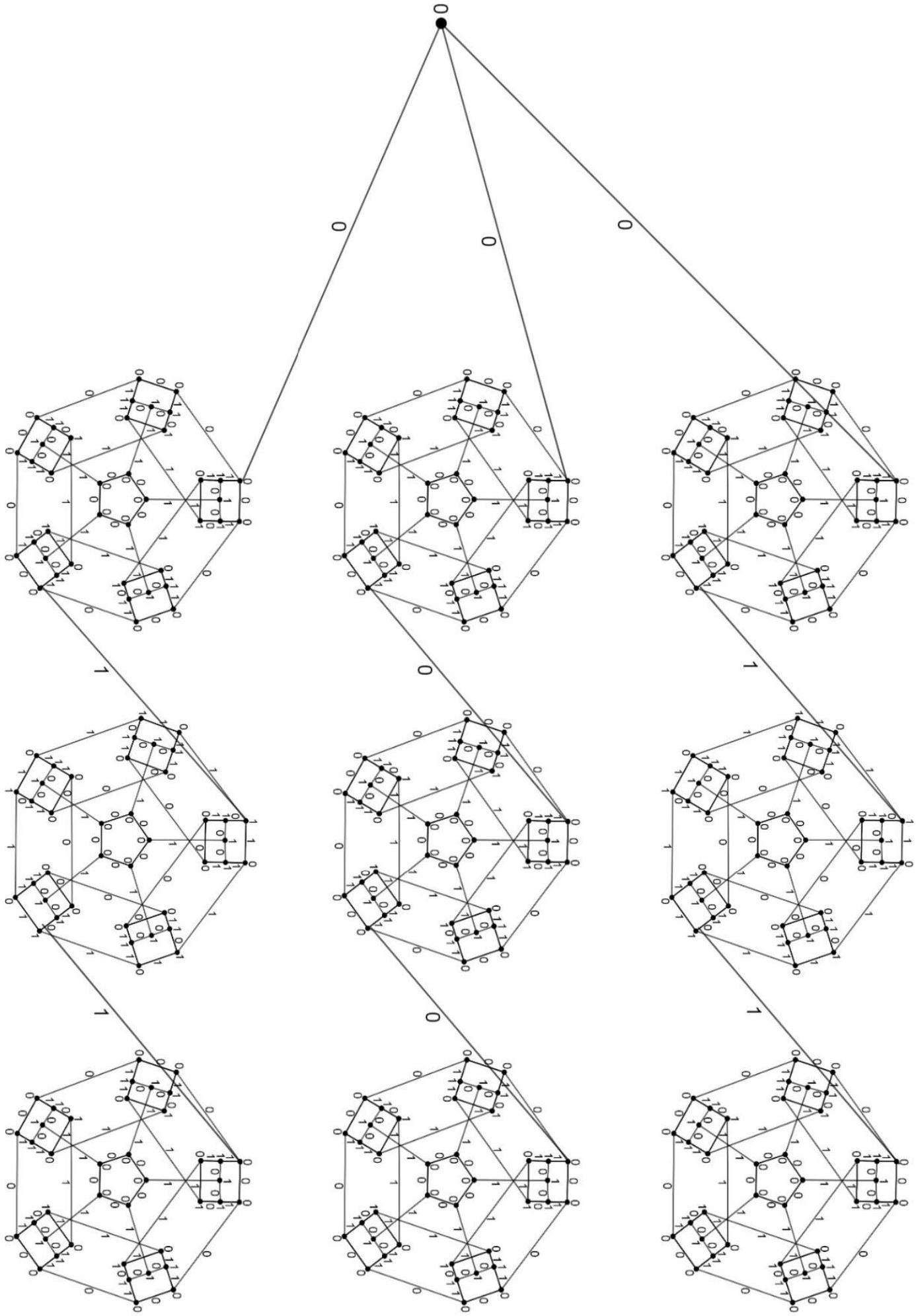}
\caption{ A Cordial Labeling of $P_{3}^{3}\left(3_{3}, G_{5}\right)$}\label{14}
\end{center}
\end{figure}
\end{proof}

\section{Conclusion}
In this paper, we have proved that the Goldberg Snark graph and its Star graph are Cordial. Further, we have shown that the Path-Union of Goldberg Snark graph and One-point union of path of Goldberg Snark graph satisfy the properties of Cordial labeling.

\nocite{*}
\bibliographystyle{abbrvnat}
\bibliography{sample-dmtcs}
\label{sec:biblio}

\end{document}